\documentclass[11pt]{article}

\usepackage{amssymb,amsthm,amsmath}
\usepackage{graphicx}
\usepackage{epsf}
\usepackage[small]{caption}

\newtheorem{theorem}{Theorem}
\newtheorem{lemma}{Lemma}
\newtheorem{proposition}{Proposition}
\newtheorem{corollary}{Corollary}

\newcommand{\eps}{\varepsilon}

\newcommand{\RR}{\mathbb{R}}

\newcommand{\conv}{{\rm conv}}

\newcommand{\etal}{{et~al.}}
\newcommand{\ie}{{i.e.}}
\newcommand{\eg}{{e.g.}}

\newcommand{\old}[1]{}

\title{Covering Paths for Planar Point Sets}

\author{
Adrian Dumitrescu\thanks{Computer Science, University of Wisconsin--Milwaukee, USA.
Email: \texttt{dumitres@uwm.edu}. Research was supported by the NSF
grant DMS-1001667.}
\and
D\'aniel Gerbner\thanks{Alfr\'ed R\'enyi Institute of Mathematics,
  Hungarian Academy of Sciences, Budapest, Hungary. Email:
  \texttt{gerbner.daniel@renyi.mta.hu}. Research was supported by OTKA,
  grant NK 78439 and by OTKA under EUROGIGA project GraDR
  10-EuroGIGA-OP-003.}
\and
Bal\'azs Keszegh\thanks{Alfr\'ed R\'enyi Institute of Mathematics,
  Hungarian Academy of Sciences, Budapest, Hungary. Email:
\texttt{keszegh.balazs@renyi.mta.hu}. Research was supported by OTKA,
grant NK 78439 and by OTKA under EUROGIGA project GraDR 10-EuroGIGA-OP-003.}
\and
Csaba D. T\'oth\thanks{Department of Mathematics, California State
  University, Northridge, USA and University of Calgary, Canada. Email:
  \texttt{cdtoth@ucalgary.ca}.
Research was supported by the NSERC grant RGPIN 35586, the NSF
grant CCF-0830734, and the Fields Institute for Research in
Mathematical Sciences, Toronto, Canada.}
}

\begin{document}

\maketitle


\begin{abstract}
Given $n$ points in the plane, a \emph{covering path} is a polygonal path
that visits all the points. If no three points are collinear, every covering
path requires at least $n/2$ segments, and $n-1$ straight line segments
obviously suffice even if the covering path is required to be noncrossing.
We show that every set of $n$ points in the plane admits a (possibly self-crossing)
covering path consisting of $n/2 +O(n/\log{n})$ straight line segments.
If the path is required to be noncrossing, we prove that
$(1-\eps)n$ straight line segments suffice for a small constant $\eps>0$,
and we exhibit $n$-element point sets that require at least $5n/9 -O(1)$
segments in every such path.
Further, the analogous question for noncrossing \emph{covering trees}
is considered and similar bounds are obtained.
Finally, it is shown that computing a noncrossing covering path for $n$ points
in the plane requires $\Omega(n \log{n})$ time in the worst case.
\end{abstract}

\section{Introduction}  \label{sec:intro}

In this paper we study polygonal paths visiting a finite set of points
in the plane.
A \emph{spanning path} is a directed Hamiltonian path
drawn with straight line edges. Each edge in the path connects two
of the points, so a spanning path can only turn at one of the given
points. Every spanning path of a set of $n$ points consists of $n-1$
segments. A \emph{covering path} is a directed polygonal path in the
plane that visits all the points. A covering path can make a turn at
any point, \ie, either at one of the given points or at a (chosen) Steiner point.
Obviously, a spanning path for a point set $S$ is also a covering path for $S$.
If no three points in $S$ are collinear, every covering path consists of at
least $\lceil n/2\rceil$ segments. A \emph{minimum-link} covering path
for $S$ is one with the smallest number of segments (links).
A point set is said to be in \emph{general position} if no three
points are collinear.

We study the following two questions concerning covering paths posed
by Mori\'c~\cite{Mo10,Mo11} as a generalization of the well-known
puzzle of linking $9$ dots in a $3 \times 3$ grid with a
polygonal path having only $4$ segments~\cite{Lo1914}.
Another problem which leads to these questions is separating red
from blue points~\cite{FKMU10}.  
\begin{enumerate} \itemsep -1pt
\item What is the minimum number, $f(n)$, such that every set of
$n$ points in the plane can be covered by a (possibly self-intersecting)
polygonal path with $f(n)$ segments?
\item What is the minimum number, $g(n)$, such that every set of
$n$ points in the plane can be covered by a \emph{noncrossing}
polygonal path with $g(n)$ segments?
\end{enumerate}

If no three points are collinear, then each segment of
a covering path contains at most two points, thus $\lceil n/2\rceil$
is a trivial lower bound for both $f(n)$ and $g(n)$.
Mori\'c conjectured that the answer to the first problem is
$n(1/2+o(1))$ while the answer to the second is $n(1-o(1))$.
We confirm his first conjecture (Theorem~\ref{T1}) but refute
the second (Theorem~\ref{paththm})\footnote{The first item was observed
by the current authors during the Canadian Conference CCCG 2010 and
was also communicated to the authors of~\cite{DO11}.}.
A consideration of these questions in retrospect appears
in~\cite{DO11}. 

\begin{theorem}\label{T1}
Every set of $n$ points in the plane admits a (possibly self-crossing)
covering path consisting of $n/2 +O(n/\log{n})$ line segments.
Consequently, $\lceil n/2\rceil \leq f(n) \leq n/2 +O(n/\log{n})$.
A~covering path with $n/2 +O(n/\log{n})$ segments
can be computed in $O(n^{1+\eps})$ time, for every $\eps>0$.
\end{theorem}

As expected, the noncrossing property is much harder to deal with.
Every set of $n$ points in the plane trivially admits a
noncrossing path consisting of $n-1$ straight line segments that
visits all the points, \eg, by sorting the points along some
direction, and then connecting them in this order. On the other hand,
again trivially, any such covering path requires at least $\lceil n/2
\rceil$ segments, if no three points are collinear.
We provide the first nontrivial upper and lower bounds for $g(n)$,
in particular disproving the conjectured relation $g(n)=n(1-o(1))$.

\begin{theorem} \label{paththm}
Every set of $n$ points in the plane admits a noncrossing covering path
with at most  $\lceil (1-1/601080391)n\rceil-1$ segments.
Consequently, $g(n) \leq \lceil (1-1/601080391)n\rceil-1$.
A noncrossing covering path with at most this many
segments can be computed in $O(n\log n)$ time.
\end{theorem}

\begin{theorem}\label{thm:path-lower}
There exist $n$-element point sets that
require at least $(5n-4)/9$ segments in any noncrossing
covering path. Consequently, $g(n) \geq (5n-4)/9$.
\end{theorem}

In the proof of Theorem~\ref{paththm}, we construct a noncrossing
covering path that can easily be extended to a noncrossing covering
cycle by adding one Steiner point (and two segments).

\begin{corollary} \label{cor:cycle}
Every set of $n\geq 2$ points in the plane admits a noncrossing covering
cycle with at most $\lceil (1-1/601080391)n\rceil+1$ segments. A noncrossing
covering cycle of at most this many segments can be
computed in $O(n\log n)$ time.
\end{corollary}

\paragraph{Covering trees.}
For covering a finite point set in the plane, certain types of geometric
graphs other than paths may also be practical. A noncrossing path or tree,
for example, are equally useful for separating a red and blue set of
points~\cite{FKMU10}, which is one of the motivating problems.
A \emph{covering tree} for a planar point set $S$ is a tree drawn in the plane
with straight-line edges such that every point in $S$ lies at a vertex or
on an edge of the tree. The lower and upper bounds $\lceil n/2\rceil \leq f(n)
\leq n/2 +O(n/\log{n})$ of Theorem~\ref{T1} trivially carry over for the number
of edges of covering trees (with possible edge crossings).

Let $t(n)$ be the minimum integer such that every set of $n$ points in the plane
admits a \emph{noncrossing covering tree} with $t(n)$ straight-line edges.
Since every path is a tree, we have $t(n)\leq g(n)\leq \lceil (1-1/601080391)\rceil$
from Theorem~\ref{paththm}. However, a noncrossing covering tree is significantly
easier to obtain than a noncrossing covering path. By simplifying the proof
of Theorem~\ref{paththm}, we derive a stronger upper bound for covering trees.

\begin{theorem}\label{thm:trees-upper}
Every set of $n$ points in the plane admits a noncrossing covering
tree with at most $\lfloor 5n/6\rfloor$ edges.
Consequently, $t(n) \leq \lfloor 5n/6\rfloor$.
A covering tree with at most $\lfloor 5n/6\rfloor$ edges can be
computed in $O(n\log n)$ time.
\end{theorem}

By modifying the lower bound analysis in the proof of Theorem~\ref{thm:path-lower},
we show that the same point set used there yields a slightly weaker lower bound
for noncrossing covering trees.

\begin{theorem}\label{thm:tree-lower}
There exist $n$-element point sets in the plane that
require at least $(9n-4)/17$ edges in any noncrossing
covering tree. Consequently, $t(n) \geq (9n-4)/17$.
\end{theorem}

Instead of minimizing the number of edges in a covering tree, one can
try to minimize the number of line segments, where each segment is
either a single  edge or a chain of several collinear edges of the
tree. Let $s(n)$ be the minimum integer
such that every set of $n$ points in the plane admits a
\emph{noncrossing covering tree} with $s(n)$ line segments. By definition,
we trivially have $s(n) \leq t(n)$. In addition, we determine an exact
formula for $s(n)$:

\begin{proposition}\label{prop:tree:s(n)}
We have
$$ s(n)= \begin{cases}
n-1 \text{~~~~~~if } n=2,3,4 \\
\lceil n/2 \rceil \text{~~~~~~if } n \geq 5.
\end{cases}
$$
\end{proposition}

\paragraph{Bicolored variants.}
Let $S$ be a bicolored set of $n$ points, with $S=B\cup R$, where $B$ and
$R$ are the set of blue and red points, respectively.

Two covering paths,
$\pi_R$ and $\pi_B$, one for the red and one for the blue points, are
\emph{mutually noncrossing} if each of $\pi_R$ and $\pi_B$ is noncrossing,
and moreover, $\pi_R$ and $\pi_B$ do not cross (intersect) each other.
A natural extension of the monochromatic noncrossing covering path
problem  is:
What is the minimum number $j(n)$ such that every bicolored
set of $n$ points in the plane can be covered by two monochromatic
mutually noncrossing polygonal paths with $j(n)$ segments in total?
Using the construction in the proof of Theorem~\ref{thm:path-lower} we
obtain the following corollary.

\begin{corollary}\label{C1}
Given a bicolored set of $n$ points, there are
two mutually noncrossing covering paths with a total of
at most $3n/2+O(1)$ segments. Such a pair of paths can be
computed in $O(n \log{n})$ time. On the other hand, there exist
bicolored sets that require at least $5n/9 -O(1)$ segments
in any pair of mutually noncrossing covering paths.
Consequently, $5n/9 -O(1) \leq j(n) \leq 3n/2+O(1)$.
\end{corollary}

Similarly, two covering trees $\tau_R$ and $\tau_B$, one for the red
and one for the blue points, are \emph{mutually noncrossing} if each
is noncrossing and $\tau_R$ and $\tau_B$ do not cross each other.
The analogous question is in this case:
What is the minimum number $k(n)$ such that every bicolored
set of $n$ points in the plane can be covered by two monochromatic
mutually noncrossing polygonal trees with $k(n)$ edges in total?
The construction in the proof of Theorem~\ref{thm:tree-lower} yields
the following corollary. 
\begin{corollary}\label{C2}
Given a bicolored set of $n$ points, there are two mutually
noncrossing covering trees with a total of at most $n$ edges.
Such a pair of trees can be computed in $O(n \log{n})$ time.
On the other hand, there exist bicolored sets that require at least
$9n/17 -O(1)$ edges in any pair of mutually noncrossing covering trees.
Consequently, $9n/17 -O(1) \leq k(n) \leq n$.
\end{corollary}

\paragraph{Computational complexity.}
We establish an $\Omega(n \log{n})$ lower bound for computing a noncrossing
covering path for a set of $n$ points in the plane.

\begin{theorem}\label{T3}
The sorting problem for $n$ numbers is linear-time reducible to the problem of
computing a noncrossing covering path for $n$ points in the plane.
Therefore, computing a noncrossing covering path for a set of
$n$ points in the plane requires $\Omega(n \log{n})$ time in the worst
case in the algebraic decision tree model of computation.
\end{theorem}

On the other hand, a noncrossing covering tree for $n$ points can be
easily computed in $O(n)$ time; see also Section~\ref{sec:conclusion}.

\paragraph{Related previous results.}
Given a set of $n$ points in the plane, the {\sc minimum-link covering path}
problem asks for a covering path with the smallest number of segments (links).
Arkin~\etal~\cite{AMP03} proved that (the decision version of)
this problem is NP-complete. Stein and Wagner~\cite{SW01} gave a
$O(\log z)$-approximation where $z$ is the maximum number of collinear
points.

Various upper and lower bounds on the minimum
number of links needed in an axis-aligned path traversing
an $n$-element point set in $\RR^d$ have been
obtained in~\cite{BBD+08,CM98,C04,KKM94}.
Approximation algorithms with constant ratio
(depending on the dimension $d$) for this problem are developed
in~\cite{BBD+08}, while some NP-hardness results have been
claimed in~\cite{EHS10}, and further revised in~\cite{J12}.
Other variants of Euclidean TSP can be found in a survey article by
Mitchell~\cite{Mi00}.

\section{Covering Paths with Possible Self-Crossings}
\label{sec:cover1}

A set $X$ of $k$ points in general position in the plane, no two on a
vertical line, is a {\em $k$-cap} ({\em $k$-cup}, respectively) if
$X$ is in (strictly) convex position and all points of $X$ lie above
(below, respectively) the line connecting the leftmost and the
rightmost point of $X$.
Similarly, caps and cups can be defined for arbitrary points (with
allowed collinearities), with $X$ being in \emph{weakly} convex position.
By slightly abusing notation, we use the same terminology when
referring to them, and distinguish them based on the underlying point
sets.

According to a classical result of Erd\H{o}s and Szekeres~\cite{ES35},
every set of at least ${2k-4 \choose k-2}+1$ points in general position
in the plane, no two on a vertical line, contains a $k$-cap or a $k$-cup.
In particular, every such set contains $k$ points in convex position;
see also~\cite{ES60,Ma02}. They also showed that this bound is the
best possible, \ie, there exist sets of ${2k-4 \choose k-2}$ points
containing no $k$-cup or $k$-cap. More generally, there exist sets
of ${k+l-4 \choose k-2}$ points containing neither $k$-cups nor $l$-caps.
While Erd\H{o}s and Szekeres originally proved the above results for
points in general position, their arguments go though verbatim
for arbitrary point sets (with allowed collinearities), and the same
quantitative bounds hold for the resulting caps or cups
(now in weakly convex position).

Following the terminology coined by Welzl~\cite{We11}, a set $S$ of $n$ points in the
plane is called {\em perfect} if it can be covered by a (possibly self-crossing)
polygonal path consisting of at most $\lceil n/2 \rceil$ segments. It
is easy to see that a cup or a cap is perfect: indeed, a suitable
covering path can be obtained by extending the odd numbered edges of
the $x$-monotone polygonal chain connecting the points (since no two points
lie on a vertical line, any consecutive pair of these edges properly intersect).

\paragraph{Proof of Theorem~\ref{T1}.}
Let $S$ be a set of $n$ points in the plane, no three of which are collinear.
Choose an orthogonal coordinate system such that no two points have
the same $x$-coordinate. By the result of Erd\H{o}s and Szekeres~\cite{ES35},
every $m$-element subset of $S$ contains a $k$-cup or a $k$-cap for some
$k=\Omega(\log{m})$. Since every such subset is perfect, it can be
covered by a path of $\lceil k/2 \rceil$ segments.

To construct a covering path, we partition $S$ into caps and cups of size
$\Omega(\log n)$ each, and a set of less than $n/\log n$ ``leftover'' points.
Set $T=S$. While $|T|\geq n/\log n$, repeatedly find a maximum-size cup or cap
in $T$ and delete those elements from $T$. Note that $\log(n/\log{n}) =
\Omega(\log{n})$, and we have found a $k$-cup or $k$-cap for some
$k=\Omega(\log n)$ in each step. Therefore, we have found $O(n/\log n)$
pairwise disjoint caps and cups in $S$, and we are left with a set $T$
of less than $n/\log n$ points.

For each $k$-cup (or $k$-cap), construct a covering sub-path with $\lceil k/2\rceil$
segments. Link these paths arbitrarily into one path, that is, append them one
after another in any order. Finally append to this path an arbitrary
spanning path of the remaining less than $n/\log{n}$ points in $T$,
with one point per turn.

A covering path for $S$ is obtained in this way.
The total number of segments in this path is $n/2 + O(n/\log{n})$, as required.
Chv\'atal and Klincsek~\cite{CK80} showed that a maximum-size cap (and cup)
in a set of $n$ points in the plane, no 3 of which are collinear, can
be found in $O(n^3)$ time. With $O(n/\log n)$ calls to their algorithm,
a covering path with $n/2 +O(n/\log n)$ segments can be
constructed in $O(n^4/\log n)$ time in the RAM model of computation.
Now if the problem can be solved in time $O(n^4/\log n)$, it can also be
solved in time $O(n^{1+\eps})$ for any $\eps>0$:
arbitrarily partition the points into $n^{1 -\eps/3}$ subsets of
$n^{\eps/3}$ points each, solve each subset separately,
move the leftover points to the next subset, then
link the paths together with one extra segment per path.
\qed

\section{Non-Crossing Covering Paths: Upper Bound}
\label{sec:noncrossing-up}

It is easy to see that for every set of $n$ points in the plane, there
is a noncrossing covering path with at most $n-1$ segments.
For example, an $x$-monotone spanning path for $n$ points has $n-1$
segments, no two crossing edges, and no Steiner points either.
In this section, we prove Theorem~\ref{paththm} and show that $(1-\eps)n$
segments suffice for some small constant $\eps>0$. In the proof
of Theorem~\ref{paththm}, however, we still use the trivial upper
bound $n-1$ for several subsets of points with the additional
constraint that the two endpoints of the path are two given points
on the boundary of a convex region containing the points in its
interior (Lemma~\ref{trivi}).

\begin{lemma}\label{trivi}
Let $X$ be a set of $n$ points in the interior of a convex region $C$,
and let $a,b$ be two points on the boundary $\partial C$ of $C$.
Then $X\cup \{a,b\}$ admits a noncrossing covering path with
$|X|+1$ segments such that its two endpoints are $a$ and $b$, and its
relative interior lies in the interior of $C$. Such a covering path
can be constructed in $O(n\log n)$ time.
\end{lemma}

We include the easy proof for completeness (a similar lemma was also
an essential tool in \cite{FKMU10}).

\begin{proof}
Let $\ell_a$ and $\ell_b$ be tangent lines at $C$ incident to $a$ and
$b$, respectively. If $\ell_a$ and $\ell_b$ are not parallel, then
let $O=\ell_a\cap \ell_b$; otherwise let $O$ be a point at
infinity corresponding to the direction of the two parallel lines
$\ell_a$ and $\ell_b$. Sort the points in $X$ in the order in which
they are encountered by a rotating sweep line from $\ell_a$ to
$\ell_b$ around $O$ (with ties broken arbitrarily). Let $\gamma$ be
the polygonal path that starts at $a$, visits the points in $X$ in the
above sweep order, and ends at $b$. The edges of $\gamma$ are pairwise
noncrossing, since they lie in interior-disjoint wedges centered at $O$.
By construction, $\gamma$ lies in the convex hull of
$X\cup \{a,b\}\subseteq {\rm int}(C)\cup \{a,b\}$, hence the relative
interior of $\gamma$ lies in the interior of $C$, as required.
\end{proof}

Before proving Theorem~\ref{paththm}, we show how to reduce the trivial bound $n-1$
on the size of noncrossing covering paths by an arbitrarily large constant,
provided that the number of points $n$ is sufficiently large.

\begin{lemma}\label{lenyeg}
Let $S$ be a set of $n$ points in the plane that contains a cap or cup
of even size $k$. Then $S$ admits a noncrossing covering path $\gamma$
with at most $n+1-\lfloor k/6\rfloor$ segments.
Furthermore, if $S$ lies in the interior of a vertical strip $H$
bounded by two vertical lines,  $h_1$ and $h_2$,
then we may require that the two endpoints of $\gamma$ lie on $h_1$ and $h_2$,
respectively, and the relative interior of $\gamma$ lie in the interior of $H$.
\end{lemma}
\begin{proof}
Let $H$, $S\subset H$, be a vertical strip bounded by two vertical lines,
$h_1$ and $h_2$, from the left and right, respectively. Assume that $S$
contains a \emph{cap} of size $k$ (the case of a \emph{cup} is analogous).
We construct a noncrossing covering path $\gamma$ with at most
$n+1-\lfloor k/6\rfloor$ segments. We may assume that $k$ is a multiple of 6
(by decreasing $k$, if necessary, without changing the value of
$\lfloor k/6\rfloor$).

Let $P=\{p_1,p_2,\ldots , p_k\}\subset S$ be a cap of size of $k$,
labeled in increasing order of $x$-coordinates. Note that $P$ admits a
covering path $\gamma_0=(q_1,q_2,\ldots,q_{k/2+1})$ of $k/2$ segments,
where each segment of $\gamma_0$ contains two consecutive points of
the cap. We may assume (by extending or shortening $\gamma_0$ if
necessary) that the endpoints of $\gamma_0$ are on the boundary of the
vertical strip $H$, that is, $q_1\in h_1$ and $q_{k/2+1}\in h_2$.
Let $q_0\in h_1$ and $q_{k/2+2}\in h_2$ be arbitrary points vertically
below the endpoints of $\gamma_0$. With this notation, the polygonal
path $(q_0q_1)\cup \gamma_0\cup (q_{k/2+1}q_{k/2+2})$ is a convex arc.
Let $s_{k/2+1}\in h_2$ be an arbitrary Steiner point above the right endpoint
of $\gamma_0$. The two endpoints of our final covering path for $S$
will be $q_0 \in h_1$ and $s_{k/2+1} \in h_2$.

In the remainder of the proof, we first construct a noncrossing
covering path for $S$ with $n+1$ segments from $q_0$ to $s_{k/2+1}$
and then modify this construction to ``save'' $\lfloor k/6\rfloor$ segments.

\begin{figure}[htbp]
\centering
    \includegraphics[width=0.8\textwidth]{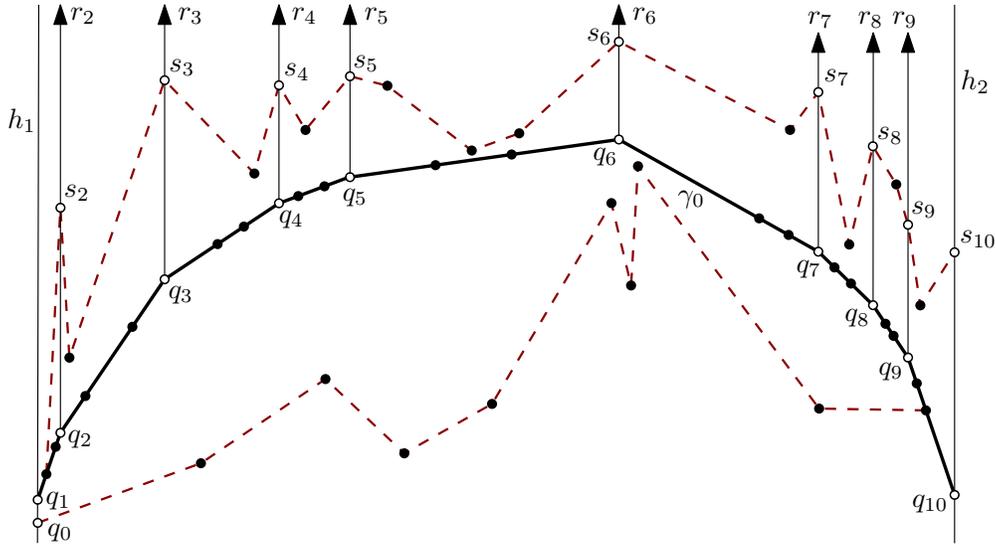}
        \caption{A set with $n=36$ points.
          A covering path $\gamma_0$ for a cap of $k=18$ points
          (bold). The region above $\gamma_0$ is decomposed into 9
          convex regions by vertical rays $r_i$, $i=2,\ldots,9$,
          each passing through a Steiner point $s_i$. We obtain a
          covering path with $n+1=37$ segments for all points by
          concatenating three covering paths.}
        \label{fig:eps1}
\end{figure}

\paragraph{Preliminary approach.}
The path $\gamma_0$ decomposes the vertical strip $H$ into two
regions: a convex region below and a nonconvex region above, with
$k/2-1$ reflex vertices at $q_2,\ldots , q_{k/2}$. Decompose the
nonconvex region above $\gamma_0$ into $k/2$ convex regions by
upward vertical rays $r_i$ emitted by $q_i$, for $i=2,\ldots,q_{k/2}$;
and place an arbitrary Steiner point $s_i$ on the ray $r_i$.

We construct a noncrossing covering path for $S$ as a concatenation of
the following three paths: The first path is a covering path
for the points in $S$ lying strictly below $\gamma_0$,
from the Steiner point $q_0\in h_1$ to $p_k \in \gamma_0$,
obtained by Lemma~\ref{trivi}.
The second path is the part of $\gamma_0$ from $p_k$ to $p_1$. The
third path consists of $k/2$ covering subpaths
for the points in $S$ lying above $\gamma_0$;
the first of these subpaths runs from $p_1$ to $s_2$, and
the others run from $s_i$ to $s_{i+1}$, for $i=1,\ldots,k/2$,
as obtained by Lemma~\ref{trivi}.

The resulting path visits all points of $S$, since the second part
visits all points along $\gamma_0$,
and the convex regions jointly contain all points below or above
$\gamma_0$. The $k/2+2$ parts of the
covering path are pairwise noncrossing, since they lie either on
$\gamma_0$ or in pairwise interior-disjoint regions whose interiors
are also disjoint from $\gamma_0$. The second part of
the covering path (the part along $\gamma_0$) covers $k$ points with
$k/2$ segments, and the remaining
$k/2+1$ parts each require one more segment than the number of points
covered. Hence, the total number of segments is $n+1-(k/2)+(k/2+1)=n+1$.

\paragraph{Modified construction.}
We now modify the above construction and ``save'' $\lfloor k/6\rfloor$ segments.
The savings come from the following two ideas. (1) It is not necessary
to decompose the entire region above $\gamma_0$ into convex pieces.
If a region above $\gamma_0$ contains all points in $S$ above $\gamma_0$,
and has fewer than $k/2$ reflex vertices, then we can decompose this
region into fewer than $k/2$ convex pieces, using fewer than $k/2$
Steiner points and thus reducing the size of the resulting covering path.
(2) If a ray emitted by a reflex vertex passes through a point of $S$
and decomposes the reflex angle into two convex angles, then a Steiner
point can be replaced by a point of $S$, which saves one segment in
the resulting covering path.

\begin{figure}[htbp]
\centering
    \includegraphics[width=0.8\textwidth]{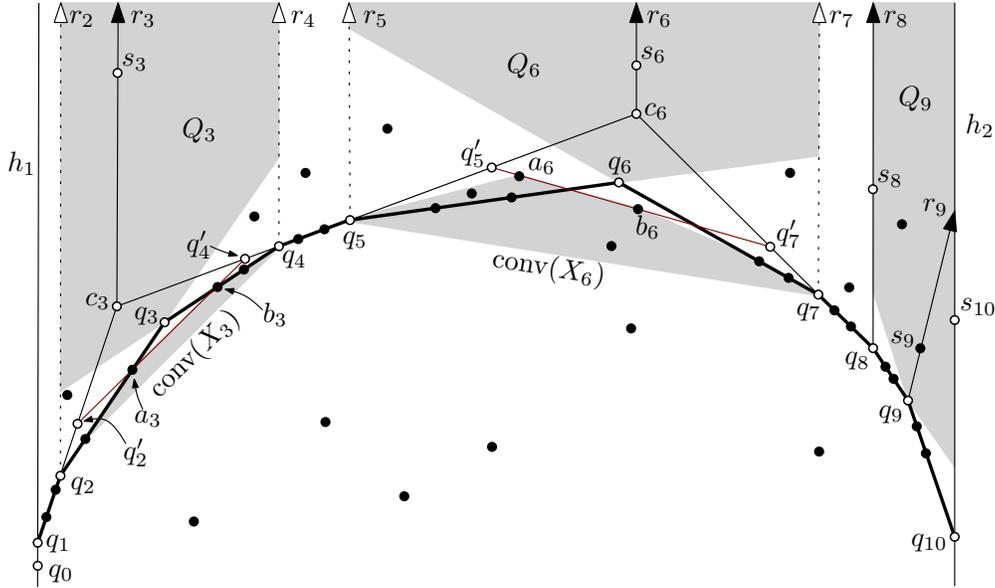}
        \caption{Local modifications for $\gamma_0$.
        Case~1: $S\cap Q_9\neq \emptyset$, we choose a ray $r_9$
        emitted by $q_9$ that passes through a point in
        $S\cap Q_9$. Case~2: $S\cap Q_3=\emptyset$, we construct ${\rm
          conv}(X_3)$, and replace the arc $(q_2,q_3,q_4)\subset
        \gamma_0$  by $(q_2,q_2',q_4',q_4)\subset \gamma_1$.
        Similarly, $S\cap Q_6=\emptyset$, we construct ${\rm
          conv}(X_6)$, and replace the arc $(q_5,q_6,q_7)\subset
        \gamma_0$  by $(q_5,q_5',q_7',q_7)\subset \gamma_1$. Note that the
         triangles $\Delta(c_3,q_2',q_4')$ and $\Delta(c_6,q_5',q_7')$
         contain no points from $S$. }
        \label{fig:eps2}
\end{figure}

We show that one of the two ideas is always applicable locally. Specifically,
we modify $\gamma_0$ by replacing some of the arcs $(q_{i-1},q_i,q_{i+1})\subset \gamma_0$,
where $i$ is a multiple of 3, by different arcs and obtain a new convex polygonal arc $\gamma_1$. 
The path $\gamma_1$ retains the property that every segment contains two
points from $S$, however it may consist of fewer segments than $\gamma_0$.
We keep the modifications ``local'' in the sense that a modification
in the neighborhood of a vertex $q_i$, where $i>0$ is a multiple of 3,
is carried out independently of modifications at all other vertices $q_j$,
where $j>0$ is a multiple of 3. Even though $\gamma_0$ may be modified in the neighborhoods 
of vertices $q_i$ and $q_{i+3}$, where $i>0$ is a multiple of 3, the intermediate edge $q_{i+1}q_{i+2}$ of $\gamma_0$ will preserved: it will be either an edge of $\gamma_1$ or contained in a longer edge of $\gamma_1$.

For every $i$, where $i$ is a positive multiple of 3, let $Q_i$ be the region above
\emph{both} lines $q_{i-1}q_i$ and $q_iq_{i+1}$, and between the vertical
rays $r_{i-1}$ and $r_{i+1}$. We distinguish two cases (refer to Fig.~\ref{fig:eps2}):

\paragraph{Case~1: $S\cap Q_i\neq \emptyset$.}
In this case, the polygonal arc $(q_{i-1},q_i,q_{i+1})$ of $\gamma_0$
is not modified but the ray $r_i$ is defined differently.
In addition, the Steiner points $s_{i-1}\in r_{i-1}$ and
$s_{i+1}\in r_{i+1}$ are defined differently.
Pick an arbitrary point $s_i\in S\cap Q_i$, and let $r_i$ be the ray emitted
by $q_i$ and passing through $s_i$. Let $r_{i-1}$ and $r_{i+1}$
be vertical rays emitted by $q_{i-1}$ and $q_{i+1}$,
respectively, like before. Decompose the region above
$(q_{i-1},q_i,q_{i+1})$ by the two vertical rays $r_{i-1}$ and
$r_{i+1}$, and then by the (possibly nonvertical) ray $r_i$.
Choose Steiner points $s_{i-1}\in r_{i-1}$ and $s_{i+1}\in r_{i+1}$,
such that they each lie on the common boundary of two consecutive
regions of the decomposition.

\paragraph{Case~2: $S\cap Q_i= \emptyset$.} In this case we modify
the polygonal arc $(q_{i-1},q_i,q_{i+1})$ of $\gamma_0$.
Let $c_i$ be the intersection point of lines $q_{i-2}q_{i-1}$ and
$q_{i+1}q_{i+2}$ (possibly $q_{i-2}=q_0$ or  $q_{i+2}=q_{k/2+2}$).
Let $S_i$ denote the set of points of $S$ in the interior of the
triangle $\Delta(q_{i-1}c_iq_{i+1})$; and let $X_i=S_i\cup \{q_{i-1},q_{i+1}\}$.
Since $Q_i$ is empty and each of the segments $q_{i-1}q_i$ and $q_iq_{i+1}$
contains two points of $S$, the convex hull of $X_i$ has at least four
vertices, \ie, $\conv(X_i)$ is not a triangle.
Let $a_ib_i$ be an arbitrary edge of $\conv(X_i)$, that is incident to
neither $q_{i-1}$ nor $q_{i+1}$. The line $a_ib_i$ intersects the sides $q_{i-1}c_i$
and $q_{i+1}c_i$ of the triangle $\Delta(q_{i-1}c_iq_{i+1})$. Denote the
intersection points by $q_{i-1}'\in q_{i-1}c_i$ and $q_{i+1}'\in q_{i+1}c_i$.
Replace $(q_{i-1},q_i,q_{i+1})$ by $(q_{i-1},q_{i-1}',q_{i+1}',q_{i+1})$
to obtain $\gamma_1$.

The segment $q_{i-1}'q_{i+1}'$ contains points $a_i,b_i\in S$. Notice
that $q_{i-1}q_{i-1}'$ is collinear with $q_{i-2}q_{i-1}$; and similarly
$q_{i+1}'q_{i+1}$ is collinear with $q_{i+1}q_{i+2}$. Therefore,
$q_{i-1}$ and $q_{i+1}$ are not vertices of $\gamma_1$. Let $r_i$ be a
vertical ray emitted by $c_i$, and pick an arbitrary Steiner point
$s_i\in r_i$.
Decompose the region above $(q_{i-1},q_{i-1}',q_{i+1}',q_{i+1})\subset
\gamma_1$ by the polygonal arc
$(q_{i-1},c_i, q_{i+1})$ and the vertical ray $r_i$. Since the
triangle $\Delta(c_i,q_{i-1}',q_{i+1}')$
contains no point from $S$, the two convex regions adjacent to $r_i$
(only one region in the extremal case $i=k/2$)
contain all the points of $S$ lying above
$(q_{i-1},q_{i-1}',q_{i+1}',q_{i+1})$.

\begin{figure}[h]
\centering
    \includegraphics[width=0.8\textwidth]{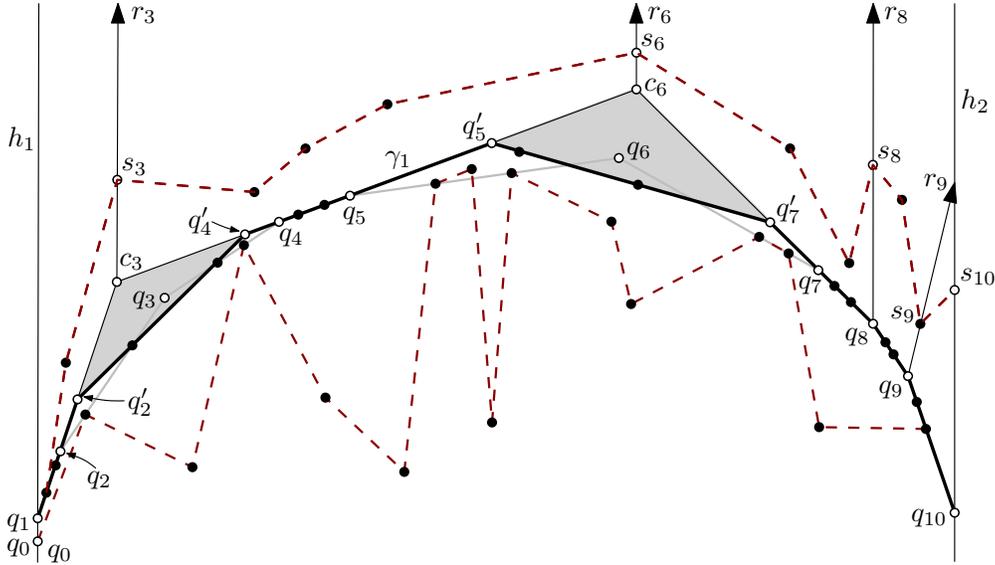}
        \caption{The set of $n=36$ points from Fig.~\ref{fig:eps1}.
        A covering path $\gamma_1$ for a cap of $14$ points (bold).
        A dashed noncrossing path below $\gamma_1$, and a dashed
        noncrossing path above $\gamma_1$ cover all remaining points
        in $S$. The resulting covering path has $n-2=34$ segments.}
        \label{fig:eps3}
\end{figure}

After the local modifications, we proceed analogously to our initial approach.
Construct a noncrossing covering path for $S$ as a concatenation of the following
three paths: A trivial path for the points strictly below $\gamma_1$ from
the Steiner point $q_0\in h_1$ to the rightmost vertex on $\gamma_1$
(Lemma~\ref{trivi});
followed by part of $\gamma_1$ from the rightmost to the leftmost point on $\gamma_1$;
the third path visits all points in the convex regions above $\gamma_1$,
passing though the points $s_i\in r_i$ between consecutive regions.

The resulting path visits all points of $S$ lying below, on, and above
$\gamma_1$ (in this order). We use $\lfloor k/6\rfloor$ fewer segments
than in our initial construction, since each local modification saves one segment:
In Case~1, we use a point $s_i\in S$ instead of a Steiner point. In
Case~2, we decrease the number of segments along $\gamma_1$ by one,
and decrease the number of relevant convex regions above $\gamma_1$ by two.
This concludes the proof of Lemma~\ref{lenyeg}.
\end{proof}

\begin{proof}[Proof of Theorem~\ref{paththm}.]
Let $S$ be a set of $n+1$ points in the plane, no three of which are collinear.
Assume, by rotating the point set if necessary, that no two points have
the same $x$-coordinate.

Lay out a raster of vertical lines in the plane such that there are
exactly $m={32\choose 16}+1=601,080,391$ points between consecutive lines;
no points on the lines or to the left of the leftmost line; and less
than $m$ points to the right of the rightmost line. By the result of
Erd\H{o}s and Szekeres~\cite{ES35}, there is a cap or cup of 18
points between any two raster lines. By Lemma~\ref{lenyeg} (for $n=m$
and $k=18$), the $m$ points between consecutive raster lines admit a
noncrossing covering path
with $m+1-3=m-2$ segments such that the two endpoints of the path are
Steiner points on the two raster lines, and the relative interior of
the path lies strictly between the raster lines. The $n_0<m$ points to
the right of the rightmost raster line can be covered by an
$x$-monotone path with $n_0$ segments starting at the (unique) Steiner
point on that line assigned by the previous group of points.

The noncrossing covering paths between consecutive raster lines can be
joined into a single noncrossing covering path for $S$ by adding one
vertical segment on each raster line except for the first and the last
one, as depicted in Fig.~\ref{pathjoin}. The total number of segments
is $n-\lfloor n/m\rfloor-1\leq \lceil (1-1/601080391)n\rceil-1$, as claimed.
\end{proof}
\begin{figure}[hbtp]
\centering
  \includegraphics[width=0.6\textwidth]{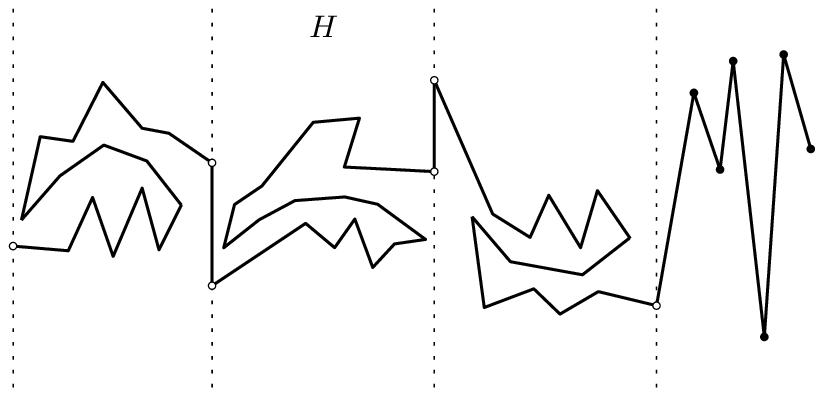}
        \caption{Joining noncrossing covering paths in the last step.}
        \label{pathjoin}
\end{figure}

\begin{proof}[Proof of Corollary~\ref{cor:cycle}.]
In the proof of Theorem~\ref{paththm}, we constructed a noncrossing
covering path $\gamma$ for $S$ such that the two endpoints of $\gamma$
are leftmost and the rightmost vertices of $\gamma$. Hence $\gamma$ can be
augmented to a noncrossing covering \emph{cycle} by adding a new
vertex of sufficiently large $y$-coordinate, and thus proving
Corollary \ref{cor:cycle}.
\end{proof}

\section{Noncrossing Covering Paths: Lower Bound}
\label{sec:cover2}

\paragraph{Proof of Theorem~\ref{thm:path-lower} (outline).}
For every $k\in \mathbb{N}$, we construct a set $S$ of $n=2k$ points
in the plane in general position, where all points are very close to
the parabola $x\rightarrow x^2$. We then show that every noncrossing
covering path $\gamma$ consists of at least $(5n-4)/9$ segments.
The lower bound is based on a charging scheme: we distinguish \emph{perfect}
and \emph{imperfect} segments in $\gamma$, containing 2 and fewer than 2 points
of $S$, respectively. We charge every perfect segment to a ``nearby''
endpoint of an imperfect segment or an endpoint of $\gamma$, such that
each of these endpoints is charged at most twice. This implies that
at most about $\frac{4}{5}$ of the segments are perfect, and the
lower bound of $(5n-4)/9$ follows. We continue with the details.

\paragraph{A technical lemma.}
We start with a simple lemma, showing that certain segments in a
noncrossing covering path are almost parallel. We say that a line
segment $s$ \emph{traverses} a circular disk $D$ if $s$ intersects the
boundary of $D$ twice.

\begin{figure}[htbp]
\centerline{\epsfxsize=0.75\textwidth \epsffile{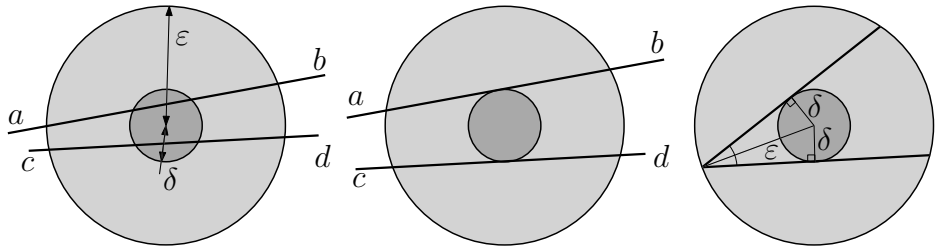}}
\caption{Left: Two noncrossing segments, $ab$ and $cd$, traverse two
concentric disks of radii $\eps > \delta >0$. Middle: The segments are
  translated to be tangent to the disk of radius $\delta$.
  Right: The angle between two adjacent chords is $2\arcsin(\delta/\eps)$.
\label{fig:parallel}}
\end{figure}

\begin{lemma}\label{lem:par}
Let $\varphi \in (0,\frac{\pi}{2})$ be an angle. For every
$\eps>0$, there  exists $\delta\in (0,\eps)$
such that if two noncrossing line segments $ab$ and $cd$ both
traverse two concentric disks of radii $\eps$ and $\delta$,
then the supporting lines of the segments $ab$ and $cd$
meet at an angle at most $\varphi$.
\end{lemma}
\begin{proof}
Let $ab$ and $cd$ be two noncrossing line segments that both traverse
two concentric disks of radii $\eps>\delta>0$. Refer to Fig.~\ref{fig:parallel}.
By translating the segments, if necessary, we may assume that both
are tangent to the disk of radius $\delta$. Clip the segments in the
disk of radius $\eps$ to obtain two noncrossing chords.
The angle between two noncrossing chords is maximal if they
have a common endpoint. In this case, they meet at an angle
$2\arcsin(\delta/\eps)$. For every $\eps>0$, $\exists \delta\in (0,\eps)$
such that $2\arcsin(\delta/\eps)<\varphi$.
\end{proof}

\paragraph{Construction.} For every $k\in \mathbb{N}$, we define
a set $S=\{a_1,\ldots , a_k, b_1,\ldots , b_k\}$ of $n=2k$ points.
Initially, let $A=\{a_1,\ldots , a_k\}$ be a set of $k$ points on the
first-quadrant part of the parabola $\alpha: x\rightarrow x^2$ such that
no two lines determined by $A$ are parallel. (To achieve strong general
position, we shall slightly perturb the points in $S$ in the last step of the
construction.) Label the points in $A$ by $a_1,\ldots , a_k$ in increasing
order of $x$-coordinates. Each point $b_i$ will be in a small
$\delta$-neighborhood of $a_i$, for a suitable $\delta>0$ and $i=1,\ldots,k$.
The pairs $\{a_i,b_i\}$ are called \emph{twin}s.
The value of $\delta>0$ is specified in the next paragraph.
See Fig.~\ref{58} for a sketch of the construction.
\begin{figure}[htbp]
\centerline{\epsfxsize=0.45\textwidth \epsffile{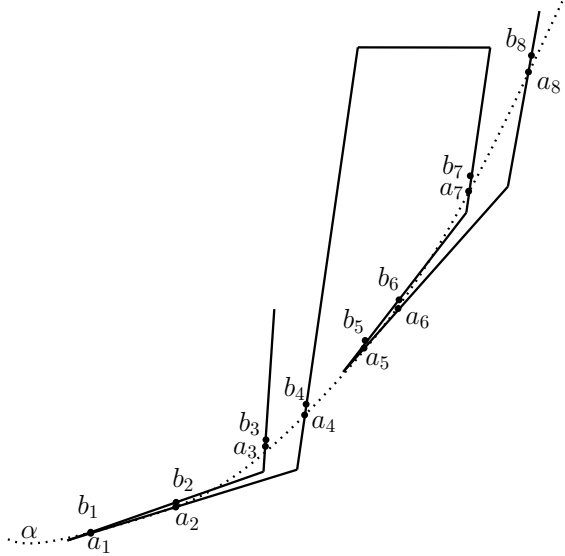}}
\caption{A sketch of our construction $S$ with $k=8$ twins.
(The figure is not to scale.) It is also indicated how 5 consecutive
  segments of a noncrossing path can cover 4 consecutive twins.}
\label{58}
\end{figure}

For every $r>0$, let $D_i(r)$ denote the disk of radius $r$ centered at $a_i\in A$.
Since the points in $A$ are in strictly convex position, points in $A$ determine
${k\choose 2}$ distinct lines. Let $(2\varphi)\in (0,\frac{\pi}{2})$ be the
minimum angle between two lines determined by $A$ (recall that no two such lines
are parallel). Let $\eps>0$ be sufficiently small such if a line intersects two
disks in $\{D_1(\eps),\ldots ,D_k(\eps)\}$, then it meets the line
passing through the centers of the two disks at an angle less than $\varphi/2$.
It follows that any line intersects at most two disks $D_1(\eps),\ldots , D_k(\eps)$
(\ie , the $\eps$-neighborhoods of at most two points in $A$). By
Lemma~\ref{lem:par}, there exists $\delta_0>0$ such that if two
noncrossing segments traverse both $D_i(\eps)$ and $D_i(\delta_0)$,
then their supporting lines meet at an angle less than $\varphi$. For
$i=1,\ldots , k-1$, let $\delta_i>0$ be the maximum distance between
the supporting line of $a_ia_{i+1}$ and points on the arc of the
parabola $\alpha$ between $a_i$ and $a_{i+1}$. We are ready to
define $\delta>0$; let $\delta=\min\{\delta_i:i=0,1,\ldots, k-1\}$.

We now choose points $b_i\in D_i(\delta)$, for $i=1,\ldots , k$, in
\emph{reverse} order. Let $\ell_k$ be a line that passes through $a_k$
such that its slope is larger than the tangent of the parabola $\alpha$ at
$a_k$. Let $b_k$ be a point in $\ell_k\cap D_k(\delta)$ above the
parabola $\alpha$. Having defined line $\ell_j$ and point $b_j$
for all $j>i$, we choose $\ell_i$ and $b_i\in \ell_i \cap D_i(\delta)$
as follows:
\begin{itemize}
\item let $\ell_i$ be a line passing through $a_i$ such that its slope is
  larger than that of $\ell_{i+1}$;
\item let $b_i\in \ell_i\cap D_i(\delta)$ be above the parabola $\alpha$; and
\item let $b_i$ be so close to $a_i$ that for every $j$, $i<j\leq k$, the
      supporting lines of segments $a_ia_j$ and $b_ib_j$ meet in the disk $D_i(\eps)$.
\end{itemize}

Write $B=\{b_i,\ldots , b_k\}$. We also ensure in each iteration that
in the set $S=A\cup B$,
(1) no three points are collinear; (2) no two lines determined by the points
are parallel; and (3) no three lines determined by disjoint pairs of
points are concurrent.

Note that $S$ is not in strong general position: for instance, all points in
$A$ lie on a parabola. (By strong general position it is meant here there is
no nontrivial algebraic relation between the coordinates of the points.)
In the last step of our construction, we slightly perturb the points in $S$.
However, for the analysis of a covering path, we may
ignore the perturbation.

Let $\gamma$ be a noncrossing covering path for $S$.
By perturbing the vertices of $\gamma$ if necessary, we may assume
that every point in $S$ lies in the relative interior of a
segment of $\gamma$. Denote by $s_0$, $s_1$ and $s_2$, respectively,
the number of segments in $\gamma$ that contain 0, 1, and 2 points
from~$S$. We establish the following inequality.

\begin{lemma}\label{pp:s01}
$s_2\leq 4(s_0+s_1+1)$.
\end{lemma}

Before the proof of Lemma~\ref{pp:s01}, we show that it
immediately implies Theorem~\ref{thm:path-lower}.

\begin{proof}[Proof of Theorem~\ref{thm:path-lower}.]
Counting the number of points incident to the segments,
we have $n=s_1+2s_2$. The number of segments in $\gamma$ is
$s_0+s_1+s_2$. This number can be bounded from below
by using Lemma~\ref{pp:s01} as follows.
\begin{eqnarray}
s_0+s_1+s_2&=& \frac{4(s_0+s_1+1)+5s_0+5s_1-4}{9}+s_2\nonumber\\
&\geq& \frac{s_2+5s_0+5s_1-4}{9}+s_2\nonumber\\
&\geq& \frac{5(s_1+2s_2)-4}{9} =\frac{5n-4}{9},\nonumber
\end{eqnarray}
as claimed
\end{proof}

For the proof of Lemma~\ref{pp:s01}, we introduce a charging scheme:
each perfect segment is charged to either an endpoint of
an imperfect segment, or one of the two endpoints of $\gamma$
such that every such endpoint is charged at most twice.
The charges will be defined for maximal $x$-monotone chains
of perfect segments. A subpath $\gamma'\subseteq \gamma$
is called \emph{$x$-monotone}, if the intersection of
$\gamma'$ with a vertical line is connected (\ie, the empty
set, a point, or a vertical segment).

Recall that all points in $A=\{a_1,\ldots  , a_k\}$ lie on
the parabola $\alpha: x\rightarrow x^2$. Let $\beta$ be
the graph of a strictly convex function that passes through the
points $b_1,\ldots , b_k$, and lies strictly above $\alpha$
and below the curve $x\rightarrow x^2+\delta$.

\paragraph{Properties of a noncrossing path covering $S$.}
We start by characterizing the perfect segments in $\gamma$.
Note that if $pq$ is a perfect segment in $\gamma$, then $pq$ contains either
a twin, or one point from each of two twins. First we make a few observations
about perfect segments containing points from two twins.

\begin{lemma}\label{pp:perfect}
Let $pq$ be a perfect segment in $\gamma$ that contains one point
from each of the twins $\{a_i,b_i\}$ and $\{a_j,b_j\}$, where $i<j$.
Then $pq$ intersects both $D_i(\delta)$ and $D_j(\delta)$,
and its endpoints lie below the curve $\beta$.
\end{lemma}
\begin{proof}
The distance between any two twin points is less than $\delta$,
so $pq$ intersects the $\delta$-neighborhood of $a_i$ and $a_j$
(even if $pq$ contains $b_i$ or $b_j$). The line $pq$
intersects the parabolas $\alpha:x\rightarrow x^2$ and
$x\rightarrow x^2+\delta$ twice each. It also intersects $\beta$
exactly twice: at least twice, since $\beta$ is between the two parabolas;
and at most twice since the region above $\beta$ is strictly convex.
All points in $\{a_i,b_i\}$ and $\{a_j,b_j\}$ are on or below $\beta$;
but $pq$ is above $\beta$ at some point between its intersections
with $\{a_i,b_i\}$ and $\{a_j,b_j\}$, since $\delta\leq \delta_i$.
Hence the endpoints of $pq$ are below $\beta$.
\end{proof}

\begin{lemma}\label{pp:nearby}
Let $pq$ be a perfect segment of $\gamma$ that contains one point from
each of the twins $\{a_i,b_i\}$ and $\{a_j,b_j\}$, where $i<j$. Assume that $p$
is the left endpoint of $pq$. Let $s$ be the segment of $\gamma$
containing the other point of the twin $\{a_i,b_i\}$. Then one of the
following four cases occurs.
\begin{itemize}
\item []\emph{Case 1:} $p$ is incident to an imperfect segment of
  $\gamma$, or $p$ is an endpoint of $\gamma$;
\item []\emph{Case 2:} $s$ is imperfect;
\item []\emph{Case 3:} $s$ is perfect, one of its endpoints $v$ lies
  in $D_i(\eps)$, and $v$ is either incident to some imperfect
  segment or it is an endpoint of $\gamma$;
\item []\emph{Case 4:} $s$ is perfect and $p$ is the common left
  endpoint of segments $pq$ and $s$.
\end{itemize}
\end{lemma}
\begin{figure}[htbp]
\centerline{\epsfxsize=0.7\textwidth \epsffile{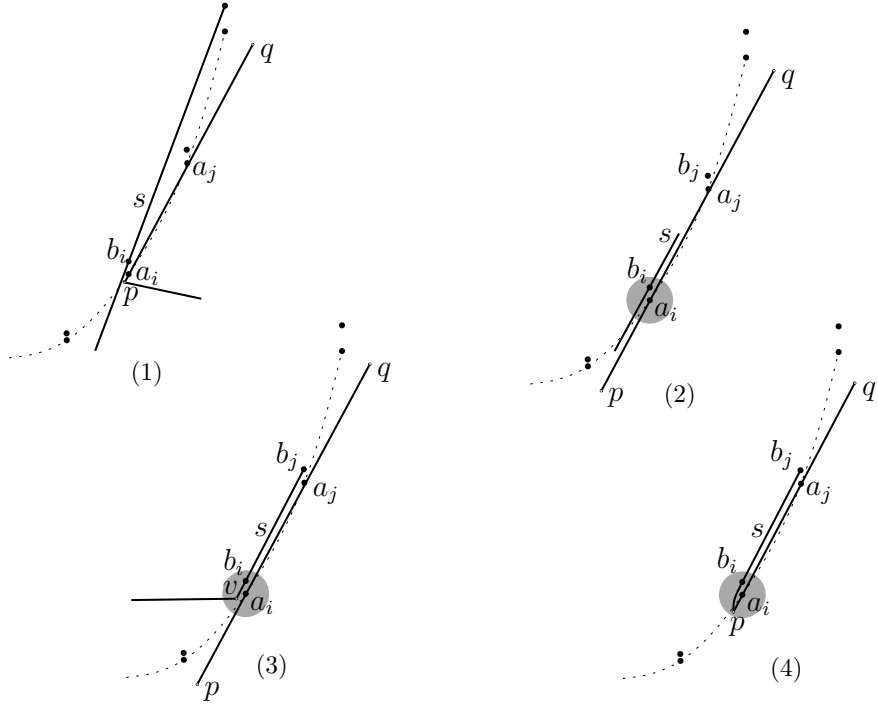}}
\caption{The four cases in Lemma~\ref{pp:nearby} for
  a perfect segment $pq$ that contains one point from each
  of the twins $\{a_i,b_i\}$ and $\{a_j,b_j\}$.
The points $a_i$, $i=1,\ldots , k$, lie on the dotted parabola
$\alpha$. The disk $D_i(\eps)$ is shaded.}
\label{95}
\end{figure}
\begin{proof}
If $p$ is incident to an imperfect segment of $\gamma$, or $p$ is an
endpoint of $\gamma$, then Case~1 occurs. Assume therefore that $p$ is
incident to two perfect segments of $\gamma$, $pq$ and $pr$. If $pr=s$,
then $p$ is the common left endpoint of two perfect segments, $pq$ and
$s$, and Case~4 occurs. If $s$ is imperfect, then Case~2 occurs.

Assume now that $pr\neq s$ and $s$ is perfect. We shall show that
Case~3 occurs. We claim that the segment $pq$ traverses
$D_i(\eps)$. It is enough to show that $p$ and $q$ lie outside
of $D_i(\eps)$. Note that $pr$ does not contain any
point from the twin $\{a_i,b_i\}$ (these points are covered
by segments $pq$ and $s$). Since $pr$ is a perfect segment, it
contains two points from
$S\setminus \{a_i,b_i\}$. By construction, every line determined by
$S\setminus \{a_i,b_i\}$ is disjoint from  $D_i(\eps)$, hence $pr$
(including $p$) is outside of $D_i(\eps)$.
Since $pq$ contains a point from $\{a_j,b_j\}$, $i<j$, point $q$ is also
outside of $D_i(\eps)$. Hence $pq$ traverses $D_i(\eps)$.

We also claim that $s$ cannot traverse $D_i(\eps)$. Suppose, to the
contrary, that $s$ traverses $D_i(\eps)$.
By Lemma~\ref{lem:par}, the supporting lines of $pq$ and $s$
meet at an angle less than $\varphi$. By the choice of $\eps$,
the supporting line of $s$ can intersect the
$\eps$-neighborhoods of $a_i$ and $a_j$ only. However, by the
choice of $b_i$, if $s$ contains one point from each of $\{a_i,b_i\}$
and $\{a_j,b_j\}$, then the supporting lines of $s$ and $pq$ intersect
in $D_i(\eps)$. This contradicts the fact that the segments of
$\gamma$ do not cross, and proves the claim.

Since $s$ does not traverse $D_i(\eps)$, it has an endpoint
$v$ in $D_i(\eps)$. If $v$ is the endpoint of $\gamma$, then Case~3 occurs.
If $v$ is incident to some other segment of $\gamma$, this segment cannot
be perfect since every line intersects the $\eps$-neighborhoods of at most
two points in $A$. Hence $v$ is incident to an imperfect segment,
and Case~3 occurs.
\end{proof}

We continue with two simple observations about perfect segments
containing twins.

\begin{lemma}\label{pp:2vertical}
The supporting lines of any two twins intersect below $\alpha$.
\end{lemma}
\begin{proof}
By construction, the supporting line of every twin has positive slope;
and $a_ib_i$ has larger slope than $a_jb_j$ if $1\leq i<j\leq k$.
Furthermore, the line $a_ib_i$ has larger slope than the tangent
line of the parabola $x\rightarrow x^2$ at $a_i$, hence $a_i$
lies above the supporting line of $a_jb_j$ for $1\leq i<j\leq k$.
It follows that the supporting lines of segments
$a_ib_i$ and $a_jb_j$ intersect below $\alpha$.
\end{proof}

\begin{lemma}\label{pp:vertical}
Let $pq$ be a perfect segment of $\gamma$ that contains a twin
$\{a_i,b_i\}$, and let $q$ be the upper (\ie , right) endpoint of $pq$.
Then either $q$ is incident to an imperfect segment of $\gamma$
or $q$ is an endpoint of~$\gamma$.
\end{lemma}
\begin{proof}
Observe that $q$ lies above $\beta$.
If $q$ is an endpoint of $\gamma$, then our proof is complete.
Suppose that $q$ is incident to segments $pq$ and $qr$ of $\gamma$.
By Lemma~\ref{pp:2vertical}, $qr$ does not contain a twin.
By Lemma~\ref{pp:perfect}, $qr$ cannot contain one point from
each of two twins, either, since then its endpoints would lie below $\beta$.
It follows that $qr$ is an imperfect segment of $\gamma$, as required.
\end{proof}

\begin{proof}[Proof of Lemma~\ref{pp:s01}.]
Let $\Gamma'$ be the set of maximal $x$-monotone chains of perfect
segments in $\gamma$. Consider a chain $\gamma'\in \Gamma'$.
By Lemma~\ref{pp:vertical}, only the rightmost segment of $\gamma'$
may contain a twin. It is possible that the leftmost segment
of $\gamma'$ contains one point from each of two twins, and the left
endpoint of $\gamma'$ is incident to another perfect segment,
which is the left endpoints of another $x$-monotone chain in $\Gamma'$.

Let $pq$ be a perfect segment of $\gamma$, and part of an $x$-monotone
chain $\gamma'\in \Gamma'$. We charge $pq$ to a point $\sigma(pq)$
that is either an endpoint of some imperfect segment or an endpoint of $\gamma$.

The point $\sigma(pq)$ is defined as follows.
If $pq$ contains a twin, then charge $pq$ to the top vertex of
$pq$, which is the endpoint of an imperfect segment or an endpoint of
$\gamma$ by Lemma~\ref{pp:vertical}. Assume now that $pq$ does not
contain a twin, its left endpoint is $p$, and it contains a point from each
of the twins $\{a_i,b_i\}$ and $\{a_j,b_j\}$, with $i<j$. We consider
the four cases presented in Lemma~\ref{pp:nearby}.

In Case~1, charge $pq$ to $p$,
which is the endpoint of an imperfect segment or an endpoint of $\gamma$.
In Case~2, charge $pq$ to the left endpoint of the imperfect segment $s$
containing a point of the twin $\{a_i,b_i\}$. In Case~3, charge $pq$
to either an endpoint an imperfect segment or an endpoint of $\gamma$
located in $D_i(\eps)$. So far, every endpoint of an imperfect segment
and every endpoint of $\gamma$ is charged at most once. Now, consider
Case~4 of Lemma~\ref{pp:nearby}. In this case, $pq$ is the leftmost
segment of $\gamma'$. If $\gamma'$ contains exactly one perfect segment,
namely $pq$, then charge $pq$ to its right endpoint, which is the endpoint
of some imperfect segment or the endpoint of $\gamma$. If $\gamma'$ contains
at least two perfect segments, then pick an arbitrary perfect segment $s$,
$s \neq pq$, from $\gamma'$. Since $s$ is not the leftmost segment of $\gamma'$,
the point $\sigma(s)$ is already defined, and we let $\sigma(pq)=\sigma(s)$.
This completes the definition of $\sigma(pq)$.

Each endpoint of $\gamma$ and each endpoint of every imperfect segment is
now charged at most twice. Since $\gamma$ and every imperfect segment
has two endpoints, we have $s_2\leq 4(s_0+s_1)+4$, as required.
\end{proof}

\paragraph{Remark.} We do not know whether the lower bound $(5n-4)/9$
for the number of segments in a minimum noncrossing covering path is
tight for the $n$-element point set $S$ we have constructed. The set
$S$ certainly has a covering path with $5n/8+O(1)$ segments. Such a
path is indicated in Fig.~\ref{58}, where 5 consecutive
segments (4 perfect and one imperfect) cover 4 consecutive twins.

\section{Noncrossing Covering Trees}

An upper bound $t(n) \leq \lceil (1-1/601080391)n\rceil$ for noncrossing
covering trees follows from Theorem~\ref{paththm}. However, the
argument can be greatly simplified while also improving the bound.

\begin{proof}[Proof of Theorem~\ref{thm:trees-upper}.]
Any set of 7 points with distinct $x$-coordinates
contains a cap or cup of 4 points, say, $a,b,c,d$, from left to
right. The $4$ points $a,b,c,d$, of a cap or cup admit a covering path with 2
segments, \ie, a 2-edge star centered at the intersection point, say
$v$, of the lines through $ab$ and $cd$, respectively.
Augment this 2-edge star covering 4 points to a 5-edge star centered at $v$
and covering all 7 points. The star is contained in the vertical strip
bounded by vertical lines incident to the leftmost and the rightmost
point, respectively.

We may assume, by rotating the point set if necessary, that no two
points have the same $x$-coordinate.
Let $p_1,p_2,\ldots,p_n$ be the points in $S$ listed in left to right order.
Decompose $S$ into groups of 7 by drawing vertical lines incident
to $p_{7+6i}$, $i=0,1,\ldots$.
Any two consecutive groups in this decomposition share a point (the
last point in group $i$ is also the first point in group $i+1$).
Thus the stars covering the groups (using $5$ edges per group)
are already connected in a tree covering all points, that
yields the claimed bound.
\end{proof}

We prove a lower bound for $t(n)$ by analyzing noncrossing covering
trees of the point set  
$S=\{a_1,\ldots , a_k, b_1,\ldots , b_k\}$, $n=2k$, 
defined in Section~\ref{sec:cover2} above. 
Let $\tau$ be a noncrossing covering tree for $S$. By perturbing
the vertices of $\tau$ if necessary, we may assume that every point
in $S$ lies in the relative interior of a segment of $\tau$. Let $s_0$,
$s_1$ and $s_2$, respectively, denote the number of
segments in $\tau$ that contain 0, 1, and 2 points from $S$;
hence $n=s_1+2s_2$. We establish the following weaker version
of Lemma~\ref{pp:s01}.

\begin{lemma}\label{lem:tree}
$s_2\leq 8(s_0+s_1)+4$
\end{lemma}

Before the proof of Lemma~\ref{lem:tree}, we show that it
directly implies Theorem~\ref{thm:tree-lower}.

\begin{proof}[Proof of Theorem~\ref{thm:tree-lower}.]
The total number of segments in $\tau$ is
\begin{eqnarray}
s_0+s_1+s_2&=& \frac{8(s_0+s_1)+4+9s_0+9s_1-4}{17}+s_2\nonumber\\
&\geq& \frac{s_2+9s_0+9s_1-4}{17}+s_2\nonumber\\
&\geq& \frac{9(s_1+2s_2)-4}{17} =\frac{9n-4}{17},\nonumber
\end{eqnarray}
where we used Lemma~\ref{lem:tree} and the fact that
the total number of points is $n=s_1+2s_2$.
\end{proof}

For the proof of Lemma~\ref{lem:tree}, we set up a charging scheme,
similar to the proof of Lemma~\ref{pp:s01}. Lemma~\ref{pp:vertical}
continues to hold for $\tau$ in place of a noncrossing covering path $\gamma$ if we
replace the ``endpoints of $\gamma$'' by the ``leaves of $\tau$.'' While the path
$\gamma$ has exactly two endpoints, the tree $\tau$ may have arbitrarily many leaves.
Therefore, the charging scheme has to be modified so that no perfect segment
is charged to the leaves of $\tau$.

\begin{proof}[Proof of Lemma~\ref{lem:tree}.]
Since $S$ is in general position, no three perfect segments have
a common endpoint. Therefore, the perfect segments of $\tau$ form disjoint paths.
Let $\Gamma$ be the set of maximal chains of perfect segments; and
let $\Gamma_x$ denote the set of maximal $x$-monotone chains of perfect
segments in $\tau$.

Choose an arbitrary vertex $r_0$ in $\tau$ as a \emph{root}, and
direct all edges of $\tau$ towards $r_0$. Every chain in $\Gamma$ is
incident to either vertex $r_0$, or to a unique outgoing imperfect
edge. Since every chain in $\Gamma$ has exactly two endpoints,
at most two vertices of a chain can have degree 1 in the tree $\tau$.

In the proof of Lemma~\ref{pp:s01}, we charged every perfect segment
$pq$ of a covering path $\gamma$ to a point $\sigma(pq)$, which was
an endpoint of an imperfect segment or the endpoint of $\gamma$. The
function $\sigma$ relied on the properties established in Lemmas
\ref{pp:perfect}--\ref{pp:vertical}. These Lemmas also hold for the
covering tree $\tau$, if we replace the endpoints of $\gamma$ by the
leaves (\ie , vertices of degree 1) in $\tau$. With this interpretation,
every perfect segment $pq$ is assigned to a point $\sigma(pq)$, which
is either an endpoint of an imperfect segment of $\tau$ or an endpoint
of a chain in $\Gamma$. We are now ready to define our charging scheme
for $\tau$. Let $pq$ be a perfect segment of $\tau$.
\begin{itemize}
\item[{\rm (i)}] If $\sigma(pq)$ is an endpoint of an imperfect segment of $\tau$,
                 then charge $pq$ to $\sigma(pq)$.
\item[{\rm (ii)}] Otherwise $\sigma(pq)$ is an end point of a chain
  $\gamma_{pq}\in \Gamma$.
                 In this case, if $\gamma_{pq}$ is incident to $r_0$,
                 then charge $pq$ to the root $r_0$ of $\tau$, else
                 charge $pq$ to the outgoing imperfect edge of $\tau$
                 incident to $\gamma_{pq}$.
\end{itemize}

Similarly to the proof of Lemma~\ref{pp:s01}, each endpoint of an
imperfect segment is charged at most twice. Since every imperfect
segment has two endpoints, rule~(i) is responsible for a total charge
of at most $4(s_0+s_1)$. Each of the two endpoints of a chain $\gamma\in \Gamma$
is charged at most twice by $\sigma$, and so the root $r_0$ and every
(directed) imperfect segment is charged at most four times. Thus rule~(ii)
is responsible for a total charge of at most $4(s_0+s_1+1)$. Altogether the
total charge assigned by rules (i) and (ii) is $s_2\leq 8(s_0+s_1)+4$, as required.
\end{proof}

\begin{proof}[Proof of Proposition~\ref{prop:tree:s(n)}.]
The case of small $n$ ($n \leq 4$) is easy to handle, so assume that $n \geq 5$.
Given $S$, compute $\conv(S)$ and let $s_1$ be a segment extension of an
arbitrary edge of $\conv(S)$; $s_1$ is long enough so that it intersects all
non-parallel lines induced by pairs of points in $S':=S \setminus S \cap s_1$.
For simplicity of exposition assume that $s_1$ is a vertical segment
with all other points in $S'$ lying left of $s_1$.
If $\conv(S')$ is a vertical segment, since $|S| \geq 5$, it is easy
to find a $3$-segment covering tree for $S$. If $\conv(S')$ is not a
vertical segment, select a non-vertical hull edge of $\conv(S')$
and extend it to the right until it hits $s_1$ and to the left
until it hits all other non-parallel lines induced by pairs of points.
Let $s_2$ be this segment extension.

Continue in a similar way on the set of remaining points,
$S'':=S' \setminus S' \cap s_2$, by choosing an arbitrary edge
of $\conv(S'')$ and extending it until it hits $s_1$ or $s_2$.
If a single point is left at the end, pick a segment incident to it
and extend it until it hits the tree made from the previously chosen segments.
Otherwise continue by extending an arbitrary hull edge of the
remaining points until it hits the tree made from the previously chosen segments.
Clearly the resulting tree covers all points and has at most
$\lceil n/2 \rceil$ segments.

For the lower bound, it is clear
that $n$ points in general position require at least $\lceil n/2 \rceil$
segments in any covering tree.
\end{proof}

\section{Bicolored Variants} \label{sec:two}

\begin{proof}[Proof of Corollary~\ref{C1}.]
For the upper bound, we proceed as follows. Assume without loss of
generality that no two points have the same $x$-coordinate (after a
suitable rotation of the point set, if needed). We have $|B|+|R|=n$, and
assume w.l.o.g. that $|B| \leq n/2 \leq |R|$. Cover the red points
by an $x$-monotone spanning path $\pi_R$, which is clearly noncrossing.
Let $B=B_1 \cup B_2$ be the partition of the blue points induced by
$\pi_R$ into points above and below the red path (remaining points are
partitioned arbitrarily). Cover the points in $B_1$ (above $\pi_R$)
by an $x$-monotone covering path: for each consecutive pair of points
in the $x$-order, extend two almost vertical rays that meet far above
$\pi_R$ without crossing $\pi_R$.
Proceed similarly for covering the points in  $B_2$ (below $\pi_R$).
Connect the two resulting blue covering paths for $B_1$ and $B_2$
by using at most $O(1)$ additional segments.

The number of segments in the red path is $|R|-1$. The number of segments
in the blue path is $2|B|+O(1)$. Consequently, since  $|B| \leq n/2$,
the two covering paths comprise at most $3n/2+O(1)$ segments.
After sorting the red and blue points along a suitable direction,
a pair of mutually noncrossing covering paths as above can be
obtained in $O(n)$ time. So the entire procedure takes $O(n \log{n})$
time.

For the lower bound, use a red and a blue copy of the point set
constructed in the proof of Theorem~\ref{thm:path-lower}, each with $n/2$ points,
so that no three points are collinear.
Since covering each copy requires at least
$(5n/9-O(1))/2$ segments in any noncrossing covering path,
the resulting $n$-element point set requires at least
$5n/9 -O(1)$ segments in any pair of mutually noncrossing covering paths.
\end{proof}

\begin{proof}[Proof of Corollary~\ref{C2}.]
For the lower bound we use two copies, red and blue, of the point-set 
from the proof of Theorem~\ref{thm:tree-lower} (which is the same as
the point-set from the proof of Theorem~\ref{thm:path-lower}). 
It remains to show the upper bound.
Assume without loss of generality that no two points have the same
$x$-coordinate. Cover the blue points by a blue star with a center
high above, and the red points by a red star with a center way below.
Obviously, each star is non-crossing, and the distinct $x$-coordinates
of the points suffice to guarantee that the two stars are mutually 
non-crossing for suitable center positions. The two centers can be 
easily computed after sorting the points in the above order.
\end{proof}

\section{Computational Complexity}\label{sec:complexity}

\paragraph{Proof of Theorem~\ref{T3}.}
We make a reduction from the sorting problem in the
algebraic decision tree model of computation.
Given $n$ distinct numbers, $x_1,\ldots,x_n$, we map them in $O(n)$
time to $n$ points on the parabola $y=x^2$: $x_i \to (x_i, x_i^2)$;
similar reductions can be found in~\cite{PS85}.
Let $S$ denote this $n$-element point set.
Since no 3 points are collinear, any covering path for $S$ has at
least $\lceil n/2\rceil +1$ vertices.
We show below that, given a noncrossing covering path of $S$ with
$m=\Omega(n)$ vertices, the points in $S$
can be sorted in left to right order in $O(m)$ time;
equivalently, given a noncrossing covering path
with $m$ vertices, the $n=O(m)$ input numbers can be sorted in
$O(m)$ time. Consequently, the $\Omega(n \log{n})$ lower bound is
then implied. Thus it suffices to prove the following.

\begin{quote}
Given a noncrossing covering path $\gamma$ of $S$ with $m$ vertices,
     the points in $S$ can be sorted in left to right order in $O(m)$ time.
\end{quote}

The boundary of the convex hull of $\gamma$ is a closed polygonal curve,
denoted $\partial \conv(\gamma)$. Melkman's algorithm~\cite{Me87} computes
$\partial \conv(\gamma)$ in $O(m)$ time. (See~\cite{Aloupis} for a
review of convex hull algorithms for simple polygons, and~\cite{BC06}
for space-efficient variants). Triangulate all faces of the plane graph
$\gamma \cup \partial \conv(\gamma)$ within $O(m)$ time~\cite{Ch91},
and let $T$ denote the triangulation. The parabola $y=x^2$ intersects
the boundary of each triangle at most 6 times (at most twice per edge).
The intersection points can be sorted in each triangle in $O(1)$ time.
So we can trace the parabola $y=x^2$ from triangle to triangle through the
entire triangulation, in $O(1)$ time per triangle, thus in $O(m)$ time overall.
Since all points of $S$ are on the parabola, one can report the sorted order
of the points within the same time.
\qed

\section{Conclusion} \label{sec:conclusion}

We conclude with a few (new or previously posed) questions and some remarks.

\begin{enumerate}
\item
It seems unlikely that every point set with no three collinear points
admits a covering path with $n/2 +O(1)$ segments. Can a lower bound of the
form $f(n)=n/2 + \omega(1)$ be established?

\item
It remains an open problem to close or narrow the gap between the lower and
upper bounds for $g(n)$, $(5n-4)/9\leq g(n)\leq \lceil(1-1/601080391)n\rceil-1$;
and for $t(n)$, $(9n-4)/17\leq t(n)\leq \lfloor 5n/6\rfloor$.

\item
Let $p(n)$ denote the maximum integer such that every set of $n$ points in
the plane has a perfect subset of size $p(n)$.
As noticed by Welzl~\cite{DO11,We11}, $p(n)=\Omega(\log{n})$ immediately
follows from the theorem of Erd\H{o}s and Szekeres~\cite{ES35}. Any
improvement in this lower bound would lead to a better upper bound
on $f(n)$ in Theorem~\ref{T1}, and thus to a smaller gap relative to the trivial
lower bound $f(n) \geq n/2$. It is a challenging question whether Welzl's lower bound
$p(n)=\Omega(\log{n})$ can be improved; see also~\cite{DO11}.

\item
It is known that the minimum-link covering path problem
is NP-complete for planar paths whose segments are unrestricted in
orientation~\cite{AMP03,KKM94}.
It is also NP-complete for axis-parallel paths in $\RR^{10}$, as shown
in~\cite{J12}. Is the minimum-link covering path problem still NP-complete for
axis-aligned paths in $\RR^d$ for $2\leq d\leq 9$?
It is known~\cite{BBD+08} that a minimum-link axis-aligned covering path in the plane
can be approximated with ratio $2$. Can the
approximation ratio of $2$ be reduced?

\item
Is the minimum-link covering path problem still NP-complete for
points in general position and arbitrary oriented paths?

\item
Is the minimum-link covering path problem still NP-complete for
points in general position and arbitrary oriented noncrossing paths?

\item
Given $n$ points ($n$ even), is it possible to compute a noncrossing
perfect matching in $O(n)$ time? Observe that such a matching
can be computed in $O(n \log{n})$ time by sorting the points along
some direction. The same upper bound $O(n \log{n})$ holds for
noncrossing covering paths and noncrossing spanning paths, and this
is asymptotically optimal by Theorem~\ref{T3}.
Observe finally that a noncrossing spanning tree can be computed in
$O(n)$ time: indeed, just take a star rooted at an arbitrary point in
the set.
\end{enumerate}

\paragraph{Acknowledgment.}
The authors are grateful to an anonymous reviewer for the
running time improvement in Theorem~\ref{T1}, and to
the members of the MIT-Tufts Computational
Geometry Research Group for stimulating discussions.

\end{document}